\documentclass[11pt]{amsart}

\usepackage{latexsym,rawfonts}
\usepackage{amsfonts,amssymb}
\usepackage{amsmath,amsthm}

\usepackage[plainpages=false]{hyperref}
\usepackage{graphicx}
\usepackage{color}

\newtheorem{theorem}{Theorem}

\newtheorem{lemma}{Lemma}

\theoremstyle{definition}
\newtheorem{definition}{Definition}

\newcommand{\beq}{\begin{equation}}
\newcommand{\eeq}{\end{equation}}
\newcommand{\beqs}{\begin{eqnarray*}}
\newcommand{\eeqs}{\end{eqnarray*}}
\newcommand{\beqn}{\begin{eqnarray}}
\newcommand{\eeqn}{\end{eqnarray}}
\newcommand{\beqa}{\begin{array}}
\newcommand{\eeqa}{\end{array}}

\DeclareMathOperator{\TR}{tr}

\newcommand{\R}{\mathbb{R}}

\newcommand{\dd}{\mathop{}\!\mathrm{d}}

\newcommand{\set}[1]{\left\{#1\right\}}
\newcommand{\norm}[1]{\left\Vert#1\right\Vert}

\newcommand{\pd}{\partial}
\newcommand{\delbar}{\overline{\nabla}}

\newcommand{\uS}{\mathbb{S}^{n-1}}

\newcommand{\MA}{Monge-Amp\`ere }

\begin{document}

\title{Smooth solutions to the Gauss image problem}

\author{Li Chen}
\address{Faculty of Mathematics and Statistics, Hubei Key Laboratory of Applied Mathematics, Hubei University, Wuhan 430062, P.R. China}
\email{chernli@163.com}

\author{Di Wu}
\address{Faculty of Mathematics and Statistics, Hubei Key Laboratory of Applied Mathematics, Hubei University,  Wuhan 430062, P.R. China}
\email{wudi19950106@126.com}

\author{Ni Xiang}
\address{Faculty of Mathematics and Statistics, Hubei Key Laboratory of Applied Mathematics, Hubei University, Wuhan 430062, P.R. China}
\email{nixiang@hubu.edu.cn}

\thanks{This research was supported by funds from Natural Science Foundation of China No.11971157.}

\date{}

\begin{abstract}
In this paper we study the the Gauss image problem, which is a
generalization of the Aleksandrov problem in convex geometry. By
considering a geometric flow involving Gauss curvature and functions
of normal vectors and radial vectors, we obtain the existence of
smooth solutions to this problem.
\end{abstract}

\keywords{
  Monge-Amp\`ere equation,
  dual Orlicz-Minkowski problem,
  Gauss curvature flow,
  Existence of solutions.
}

\subjclass[2010]{
35J96, 52A20, 53C44.
}

\maketitle
\vskip4ex

\section{Introduction}

Let $K\subset \mathbb{R}^n$ be a convex body which contains the
origin in its interior and $x \in
\partial K$ be a boundary point, then the normal cone of $K$ at $x$
is defined by
\begin{equation*}
\mathcal{N}(K, x)=\{v \in \uS: \langle y-x, v\rangle\leq 0 \quad
\text{for all} \quad y \in K\},
\end{equation*}
where $\langle y-x, v\rangle$ denotes the standard inner product of
$y-x$ and $v$ in $\mathbb{R}^n$. For $\omega \subset \uS$, the
radial Gauss image of $\omega$ is defined by
\begin{equation*}
\alpha_{K}(\omega)=\bigcup_{x \in \rho_K(\omega)}\mathcal{N}(K,
x)\subset \uS,
\end{equation*}
where $\rho_K: \mathbb{S}^{n-1}\rightarrow \partial K$ is the radial
function of $\partial K$(see Section 2 for the definition).
Recently, Boroczky, Lutwak, Yang, Zhang and Zhao \cite{B2020}
proposed the Gauss image problem which link two given submeasures
via the radial Gauss image of a convex body:

\textbf{The Gauss image problem.} \emph{Suppose $\lambda$ is a submeasure
defined on the Lebesgue measurable subsets of $\uS$, and $\mu$ is a
Borel submeasure on $\uS$. What are the necessary and sufficient
conditions, on $\lambda$ and $\mu$, so that there exists a convex
body $K$ such that
\begin{equation} \label{GIP}
\lambda(\alpha_{K}(\cdot))=\mu
\end{equation}
on the Borel subsets of $\uS$? And if such a body exists, to what
extent is it unique?}

When $\lambda$ is spherical Lebesgue measure, the Gauss image
problem is just the classical Aleksandrov problem. It is necessary
to contrast the Gauss image problem with the various Minkowski
problems and dual Minkowski problems that have been extensively
studied, see \cite{BLYZ.JAMS.26-2013.831,
  CLZ.TAMS.371-2019.2623,
  CW.Adv.205-2006.33,
  HLYZ.DCG.33-2005.699,
  JLW.JFA.274-2018.826,
  JLZ.CVPDE.55-2016.41,
  Lu.SCM.61-2018.511,
  Lu.JDE.266-2019.4394,
  LW.JDE.254-2013.983,
  Lut.JDG.38-1993.131,
  LYZ.TAMS.356-2004.4359,
  Sta.Adv.167-2002.160,
   Zhu.Adv.262-2014.909,Zhu.JDG.101-2015.159}
for the $L_p$-Minkowski problem, \cite{BHP.JDG.109-2018.411,
  HP.Adv.323-2018.114,
  HJ.JFA.277-2019.2209,
  HLYZ.Acta.216-2016.325,
  LSW.JEMSJ.22-2020.893,
  Zha.CVPDE.56-2017.18,
  Zha.JDG.110-2018.543}
for the dual Minkowski problem,
\cite{BF.JDE.266-2019.7980,
  CHZ.MA.373-2019.953,
  CCL,
  HLYZ.JDG.110-2018.1,
  HZ.Adv.332-2018.57,
  LLL,LYZ.Adv.329-2018.85}
for the $L_p$ dual Minkowski problem,
\cite{BBC.AiAM.111-2019.101937,
  HLYZ.Adv.224-2010.2485,
  HH.DCG.48-2012.281,
  JL.Adv.344-2019.262} for the Orlicz Minkowski problem,
\cite{GHW+.CVPDE.58-2019.12,GHXY.CVPDE.59-2020.15, LL.TAMS.373-2020.5833} for the dual Orlicz Minkowski problem. In the
Gauss image problem, a pair of submeasures is given and it is asked
if there exists a convex body ``linking" them via its radial Gauss
image. However, in a Minkowski problem, only one measure is given,
and the question asks if this measure is a specific geometric
measure of a convex body.

To statement the solutions to the Gauss image problem  in
\cite{B2020}. We introduce some concepts (see \cite{B2020} for
details). If $\omega\subset \uS$ is contained in a closed
hemisphere, then the polar set $\omega^*$ is defined by
\begin{equation*}
\omega^*=\{v \in \uS: \langle u, v\rangle\leq 0 \quad \text{for all}
\quad u \in \omega\}.
\end{equation*}

\begin{definition}
Two Borel measures $\mu$ and $\lambda$ on $\uS$ are called
Aleksandrov related if
\begin{equation} \label{Ar}
\lambda(\uS)=\mu(\uS)>\lambda(\omega^*)+\mu(\omega)
\end{equation}
for any compact, spherically convex set $\omega \in \uS$.
\end{definition}

Note that $\lambda(\uS)=\mu(\uS)$ is obvious for a solution to
\eqref{GIP}. The following existence result for solutions to the
Gauss image problem was proved in \cite{B2020}.

\begin{theorem}\label{thm0}
Suppose $\lambda, \mu$ are Borel measures on $\uS$ and $\lambda$ is
absolutely continuous. If $\mu$ and $\lambda$ are Aleksandrov
related, then there exists a body $K$ containing the origin in its
interior such that $\lambda(\alpha_{K}(\cdot))=\mu$.
\end{theorem}

Note that for the special case in which $\mu$ is a measure that has
a density with repect to the spherical Lebesgue measure, say $f$, and $\lambda$ is a measure
that has a density with repect to the spherical Lebesgue measure,
say $g$. In this case, $\mu$ and $\lambda$ on $\uS$ are
Aleksandrov related if
\begin{equation}\label{cond1}
\int_{\uS}f=\int_{\uS}g>\int_{\omega}f+\int_{\omega^*}g
\end{equation}
for any compact, spherically convex set $\omega \in \uS$. Moreover,
the geometric problem \eqref{GIP} is the equation of \MA
type
\begin{equation} \label{dOMP-f}
g\bigg(\frac{\nabla h+hx}{|\nabla h+hx|}\bigg) |\nabla h+h x|^{-n}h
\det(\nabla^2h +hI)=f \quad \text{ on } \quad \uS,
\end{equation}
where $h$ is the support function of the polar body $K^*$ of $K$
which is defined as
\begin{equation*}
K^*=\{x \in \mathbb{R}^n: \langle x, y\rangle\leq 1 \quad \text{for
all} \quad y \in K\}.
\end{equation*}
Here $\nabla$ is the covariant derivative with respect to an
orthonormal frame on $\uS$, $I$ is the unit matrix of order $n-1$,
and $\nabla h(x) +h(x) x$ is just the point on $\partial K^*$ whose
outer unit normal vector is $x\in\uS$.

In this paper we study the existence of smooth solutions to the
equation \eqref{dOMP-f}. We obtain the following existence result.
\begin{theorem}\label{thm1}
Suppose that $f$ and $g$ are two positive smooth functions on $\uS$
. If $f$ and $g$ satisfy the condition \eqref{cond1}, then there
exists a smooth solution to the equation \eqref{dOMP-f}.
\end{theorem}

The proof of Theorem \ref{thm1} is inspired by
\cite{LSW.JEMSJ.22-2020.893}, where the existence of smooth
solutions to the Aleksandrov and dual Minkowski problem was obtained
by studying a generalized Gauss curvature flow. In fact, the Gauss
curvature flow and its various generalizations have been extensively
studied by many scholars; see for example \cite{And.IMRN.1997.1001,
AGN.Adv.299-2016.174, BCD.Acta.219-2017.1, BIS.AP.12-2019.259,CHZ.MA.373-2019.953,
CW.AIHPANL.17-2000.733, GL.DMJ.75-1994.79, Ger.CVPDE.49-2014.471,
Ham.CAG.2-1994.155, Iva.JFA.271-2016.2133,
LL.TAMS.373-2020.5833, Urb.JDG.33-1991.91} and the references
therein.

Let $M_{0}$ be a smooth, closed, uniformly convex hypersurface in
$\R^{n}$, which contains the origin in its interiors and given by a
smooth embedding $X_{0}: \mathbb{S}^{n-1} \rightarrow \R^{n}$. We
consider a family of closed hypersurfaces $\set{M_t}$ given by
$M_t=X(\uS,t)$, where $X: \mathbb{S}^{n-1}\times[0,T) \rightarrow
\R^{n}$ is a smooth map satisfying the following initial value
problem:
\begin{equation}\label{flow}
\left\{
\begin{aligned}
    \frac{\pd X}{\pd t} (x,t)
    &= -\frac{f(\nu)}{g\Big(\frac{X}{|X|}\Big)}|X|^{n}\mathcal{K} \nu +  X,\\
    X(x,0) &= X_{0}(x).
  \end{aligned}
\right.
\end{equation}
Here $\nu$ is the unit outer normal vector of the hypersurface $M_t$
at the point $X(x,t)$, $\mathcal{K}$ is the Gauss curvature of
$M_{t}$ at $X(x,t)$, and $T$ is the maximal time for which the
solution exists. We obtain the long-time existence and convergence
of the flow \eqref{flow}.

\begin{theorem}\label{thm2}
Suppose $f$ and $g$ satisfy the assumptions of Theorem \ref{thm1}.
Let $M_{0}$ be a smooth, closed, uniformly convex hypersurface  in
$\R^{n}$, which contains the origin in its interior. Then, the flow
\eqref{flow} has a unique smooth solution, which exists for all time
$t>0$. Moreover, when $t\to\infty$, a subsequence of
$M_{t}=X(\uS,t)$ converges in $C^{\infty}$ to a smooth, closed,
uniformly convex hypersurface, whose support function is a smooth
solution to the equation \eqref{dOMP-f}.
\end{theorem}

This paper is organized as follows. In section 2, we give some basic
knowledge about convex hypersurfaces and the flow \eqref{flow}. In
section 3, more properties of the flow \eqref{flow} will be proved,
based on which we can obtain the uniform lower and upper bounds of
support functions of $\set{M_t}$ via delicate analyses. In the last
section, the long-time existence and convergence of the flow
\eqref{flow} will be proved, which completes the proofs of our
theorems.

\section{Preliminaries}

\subsection{Basic properties of convex hypersurfaces}

We first recall some basic properties of convex hypersurfaces in
$\R^n$; see \cite{Urb.JDG.33-1991.91} for details.
Let $M$ be a smooth, closed, uniformly convex hypersurface in $\R^n$ enclosing
the origin.
The support function $h$ of $M$ is defined as
\begin{equation} \label{h0}
h(x) := \max_{y\in M} \langle y,x \rangle, \quad \forall\, x\in\uS,
\end{equation}
where $\langle \cdot,\cdot \rangle$ is the standard inner product in $\R^n$.

The convex hypersurface $M$ can be recovered by its support function $h$.
In fact, writing the Gauss map of $M$ as $\nu_M$, we parametrize $M$ by
$X : \uS\to M$ which is given as
\begin{equation*}
X(x) =\nu_M^{-1}(x), \quad \forall\,x\in\uS.
\end{equation*}
Note that $x$ is the unit outer normal vector of $M$ at $X(x)$.
On the other hand, one can easily check that the maximum in the definition
\eqref{h0} is attained at $y=\nu_M^{-1}(x)$, namely
\begin{equation} \label{h}
h(x) = \langle x, X(x)\rangle, \quad\forall\, x \in
\mathbb{S}^{n-1}.
\end{equation}
Let $e_{ij}$ be the standard metric of the unit sphere
$\mathbb{S}^{n-1}$, and $\nabla$ be the corresponding connection on
$\mathbb{S}^{n-1}$. Then, it is easy to check that
\begin{equation}\label{Xh}
  X(x) = \nabla h(x) + h(x)x, \quad \forall\,x\in\uS.
\end{equation}

By differentiating \eqref{h} twice, the second fundamental form $A_{ij}$ of $M$
can be also computed in terms of the support function:
\begin{equation} \label{A}
  A_{ij} =  \nabla_i\nabla_{j}h + he_{ij},
\end{equation}
where $\nabla_{i}\nabla_j$ denotes the second order covariant
derivative with respect to $e_{ij}$. The induced metric matrix
$g_{ij}$ of $M$ can be derived by Weingarten's formula:
\begin{equation} \label{g}
  e_{ij} = \langle \nabla_{i}x, \nabla_{j}x\rangle  = A_{ik}A_{lj}g^{kl}.
\end{equation}
The principal radii of curvature are eigenvalues of the matrix $b_{ij} =
A^{ik}g_{jk}$.
When considering a smooth local orthonormal frame on $\uS$, by virtue of
\eqref{A} and \eqref{g}, we have
\begin{equation} \label{bij}
  b_{ij} = A_{ij} = \nabla_{i}\nabla_jh + h\delta_{ij}.
\end{equation}
In particular, the Gauss curvature of $M$ at $X(x)$ is given by
\begin{equation*}
\mathcal{K}(x) = [\det(\nabla_{i}\nabla_jh + h\delta_{ij})]^{-1}.
\end{equation*}
We shall use $b^{ij}$ to denote the inverse matrix of $b_{ij}$.

The radial function $\rho$ of the convex hypersurface $M$ is defined as
\begin{equation*}
\rho(u) :=\max\set{\lambda>0 : \lambda u\in M}, \quad\forall\, u\in\uS.
\end{equation*}
Note that $\rho(u)u\in M$.
The Gauss map $\nu_M$ can be computed as
\begin{equation*}
  \nu_M(\rho(u)u) = \frac{\rho(u)u - \nabla \rho}{\sqrt{\rho^{2}+|\nabla \rho|^{2}}}.
\end{equation*}
If we connect $u$ and $x$ through the following equality:
\begin{equation}\label{eq:8}
\rho(u)u =X(x) =\nabla h(x) +h(x)x =\delbar h(x),
\end{equation}
where $\overline{\nabla}$ is the standard connection of
$\mathbb{R}^n$, then we have the following relations
\begin{equation}\label{eq:9}
x= \frac{\rho(u)u - \nabla \rho}{\sqrt{\rho^{2}+|\nabla
\rho|^{2}}},\quad  u= \frac{\nabla h +h(x)x}{\sqrt{|\nabla h|^2
+h^2}},
\end{equation}
and
\begin{equation}\label{xu}
\frac{h(x)}{\mathcal{K}(x)}dx=\rho^n(u)du.
\end{equation}

\subsection{Geometric flow and its associated functional}

Recalling the evolution equation of $X(x,t)$ in the geometric flow \eqref{flow},
and using similar computations as in \cite{Urb.JDG.33-1991.91}, we obtain the
evolution equation of the corresponding support function $h(x,t)$:
\begin{equation}\label{ht}
  \frac{\pd h}{\pd t} (x,t)
  = -\frac{f(x)\rho^n(u)}{g(u)} \mathcal{K}(x,t) +  h(x,t)
  \ \text{ in }\ \uS\times(0,T).
\end{equation}
Since $M_t$ can be recovered by $h(\cdot,t)$, the flow \eqref{ht} is equivalent
to the original flow \eqref{flow}.

Denote the radial function of $M_t$ by $\rho(u,t)$. For any $t$, let
$u$ and $x$ be related through the following equality:
\begin{equation*}
\rho(u,t)u =\delbar h(x,t) =\nabla h(x,t) +h(x,t)x.
\end{equation*}
Therefore, $x$ can be expressed as $x=x(u,t)$. By a direction
computation (see \cite{CHZ.MA.373-2019.953}), we have
\begin{equation}\label{eq:4}
  \frac{1}{\rho(u,t)}\pd_t\rho(u,t)
  =\frac{1}{h(x,t)}\pd_th(x,t).
\end{equation}
Now by virtue of \eqref{ht} and \eqref{eq:4}, we obtain the
evolution equation of $\rho(u,t)$:
\begin{equation}\label{rhot}
  \frac{\pd \rho}{\pd t} (u,t)
  = -\frac{f(x)\rho^{n+1}(u,t)}{g(u)h(x,t)} \mathcal{K}(x,t) +  \rho(u,t)
  \ \text{ in }\ \uS\times(0,T),
\end{equation}
where $x=x(u,t)$ is the unit outer normal vector of $M_t$ at the point $\rho(u,t)u$.

Consider the following functional:
\begin{equation}\label{Jt}
J(t)=\int_{\uS}f(x)\log h(x, t)  \dd x-\int_{\uS}g(u)\log \rho(u, t)
\dd u, \quad t\geq 0,
\end{equation}
which will turn out to be monotonic along the flow \eqref{ht}.

\begin{lemma}\label{lem02}
$J(t)$ is non-increasing along the flow \eqref{ht}. Namely $\frac{d}{dt}J(t)\leq 0$, and the equality holds if and only if $M_t$
satisfies the elliptic
equation \eqref{dOMP-f}.
\end{lemma}

\begin{proof}
Using \eqref{xu} and \eqref{eq:4}, we have
\begin{eqnarray}\label{eq:23}
\frac{d}{d t}J(t) &=&\int_{\uS} \frac{f \pd_t
h}{h}dx-\int_{\uS}\frac{g \pd_t\rho}{\rho}du\\ \nonumber &=&
\int_{\uS} \frac{\pd_t h}{h} (f-\frac{gh}{\mathcal{K}\rho^n})dx\\
\nonumber &=& \int_{\uS}
\frac{1}{h}(f-\frac{gh}{\mathcal{K}\rho^n})(h-\frac{f\mathcal{K}\rho^n}{g})
dx\\ \nonumber &=& -\int_{\uS}
\frac{1}{gh\mathcal{K}\rho^n}(gh-f\mathcal{K}\rho^n)^2 dx \\
\nonumber &\leq&0.
\end{eqnarray}
Clearly $\frac{d}{dt}J(t)=0$ if and only if
$$gh=f\mathcal{K}\rho^n.$$
Namely $M_t$ satisfies \eqref{dOMP-f}. This completes the proof.
\end{proof}

\begin{lemma}\label{lem03}
Assume that
\begin{equation*}
\int_{\uS} f(x)dx=\int_{\uS}g(u)du,
\end{equation*}
then the log-volume of $M_t$
\begin{equation}\label{log-V}
V_g(M_t)=\int_{\uS}g(u)\log \rho(u, t)
\dd u,
\end{equation}
remain unchanged under the flow \eqref{ht}.
\end{lemma}

\begin{proof}
Using \eqref{xu} and \eqref{rhot}, we have
\begin{eqnarray*}
\frac{d}{dt}V_g(M_t) &=& \int_{\uS}\frac{g\pd_t \rho}{\rho}du\\ \nonumber
&=& \int_{\uS}\frac{g}{\rho}(\rho-\frac{f\rho^{n+1}}{gh}\mathcal{K})du\\ \nonumber
&=& \int_{\uS} g(u)du-\int_{\uS}f(x)\frac{\rho^{n}\mathcal{K}}{h}du\\ \nonumber
&=& \int_{\uS} g(u)du-\int_{\uS} f(x)dx\\ \nonumber
&=&0.
\end{eqnarray*}
So, we complete the proof.
\end{proof}

\section{Uniform bounds of support functions}

In this section, we will derive uniformly positive lower and upper
bounds of support functions along the flow \eqref{flow}. Our idea
comes from the proof of Lemma 3.2 in \cite{LSW.JEMSJ.22-2020.893}.

\begin{lemma}\label{lem04}
Suppose $f$ and $g$ satisfy the assumptions in Theorem \ref{thm1}.
Let $X(\cdot, t)$ be a strictly convex solution to the flow
\eqref{flow} which encloses the origin for $t \in [0,T) $, then
exists a positive constant $C$ depending only on $M_0$, $f$ and $g$,
such that for every $t\in[0,T)$,
\begin{equation}\label{h1}
1/C \leq h(\cdot,t) \leq C \quad \text{ on } \ \uS,
\end{equation}
and
\begin{equation}\label{rho1}
1/C \leq \rho(\cdot,t) \leq C \quad \text{ on } \ \uS.
\end{equation}
\end{lemma}

\begin{proof}
Since $f$ and $g$ satisfy the condition \eqref{cond1}, the measure
$\mu$ with the density $f$, and the measure $\lambda$ with the
density $g$ are Aleksandrov related. Then, using Theorem \ref{thm0},
there exists a body $N^*$ containing the origin in its interior
satisfies the equation \eqref{GIP}. Let $N$ be the polar dual of
$N^*$. Choosing the constants $s_1>s_0>0$ such that
\begin{equation*}
N_0=s_0 N\subset K_0\subset s_1 N=N_1.
\end{equation*}
Let $r_0$ and $r_1$ be respectively the radial functions of $N_0$
and $N_1$. Clearly, $sN$ is a stationary solution to \eqref{flow} in
the generalised sense. We firstly prove that
\begin{equation*}
K_t\subset N_1
\end{equation*}
for all $t>0$ by a contradiction. Otherwise, there exists a first
time $t_0>0$ such that
\begin{equation*}
\sup_{u\in\uS}\frac{\rho(u, t_0)}{r_1(u)}=1.
\end{equation*}
Set
\begin{equation*}
P=M_{t_0}\cap N_1,
\end{equation*}
which can be a point or a closed set. Clearly, the unit normal
vector of $M_{t_0}$ coincide with that of $N_1$ for any $p \in P$.
Namely, $\nu_{M_{t_0}}(p)=\nu_{N_1}(p)$ for any $p \in P$. Moreover,
replacing $r_1$ by $(1+a)r_1$ for a small constant $a$, we may
assume that
\begin{equation*}
\frac{\partial}{\partial t} \rho(u, t)>0 \quad \text{on} \quad
P\times \{t_0\}
\end{equation*}
and also in a neighbourhood of $P\times \{t_0\}$. There exists
sufficiently small constants $\epsilon, \delta>0$ such that
\begin{equation*}
\frac{\partial}{\partial t} \rho(u, t)>\delta
\end{equation*}
for all $u \in E=\{\xi \in \uS: \rho(\xi,
t)>(1-\epsilon)r_1(\xi)\}$. Since $\nu_{M_{t_0}}(p)=\nu_{N_1}(p)$
for any $p \in P$, making $\epsilon$ small again, we have
$\nu_{N_1}(u)\approx \nu_{M_{t_0}}(u)$ for $u \in E$. Thus, using
the equation \eqref{rhot}, the Gauss curvature of $M_{t_0}$
satisfies
\begin{eqnarray*}
\mathcal{K}(M_{t_0})&<&\frac{(\rho(u,
t_0)-\delta)g(u)}{f(\nu_{M_{t_0}})\rho^{n+1}(u,
t_0)}\\&<&\frac{(r_1(u, t_0)-\delta)g(u)}{f(\nu_{N_1})(r_1(u,
t_0)-\epsilon)^{n+1}}\\&<&\frac{1}{(1-\epsilon)^n}\frac{g(u)}{f(\nu_{N_1})(r_1(u,
t_0))^{n}}\\&<& \mathcal{K}((1-\epsilon)N_1).
\end{eqnarray*}
Namely, the Gauss curvature of $M_{t_0}$ is strictly small than that
of $(1-\epsilon)N_1$ for all $\xi \in E$. Applying the comparison
principle for generalised solutions to the elliptic \MA equation
(see Theorem 1.4.6 in \cite{Gu}) to the functions $\rho(u, t_0)$ and
$(1-\epsilon)r_1(u)$, we reach a contradiction. Similarly, we can
prove that $N_0\subset M_t$ for all $t>0$.
\end{proof}

From Theorems 2 and 3 in \cite{B2020}, we know that if $f$ and $g$
are even functions satisfying
\begin{equation}\label{even}
\int_{\uS} f(x)dx=\int_{\uS}g(u)du,
\end{equation}
then $f$ and $g$ satisfy the condition \eqref{cond1}. In this case,
if $M_0$ is origin-symmetric, we can give a proof of Lemma
\ref{lem04} without using Theorem \ref{thm0}.

\begin{lemma}\label{lem04-}
Suppose that $M_0$ is origin-symmetric, $f$ and $g$ are two smooth,
positive even functions satisfying the condition \eqref{even}, then
the conclusions in Lemma \ref{lem04} hold true.
\end{lemma}

\begin{proof}
Note that $M_0$ is origin-symmetric, $f$ and $g$ are even functions,
thus $M_t$ is origin-symmetric and $h(x, t)$ is an even function.
For fixed $t\in[0,T)$, assume $h(x,t)$ attains its maximum at $x_t$.
Since $h(x, t)$ is an even function, we have by the definition of
the support function \eqref{h0}
\begin{equation}\label{eq:26}
h(x,t)\geq |\langle x, x_t\rangle| h(x_t, t).
\end{equation}
By Lemmas \ref{lem02} and \ref{lem03},
\begin{eqnarray*}
\int_{\uS} f(x)\log h(x,0)dx &\geq& \int_{\uS} f(x)\log h(x,t)dx \\
&\geq& \int_{\uS} f(x)\log |\langle x,x_t\rangle| h(x_t,t)dx \\
&\geq&C\log h(x_t,t)-C,
\end{eqnarray*}
which implies that
\begin{equation*}
\max_{\uS\times [0,T)} h(x,t)\leq C
\end{equation*}
for some positive constant $C$.

The positive lower bound of $h$ will be proved by contradiction. Let
$\{t_{k}\}\subset [0,T)$ be a sequence such that
\begin{equation*}
\min_{\uS}h(\cdot,t_k)\rightarrow 0 \quad \text{as}\quad k\rightarrow\infty.
\end{equation*}
Let $K_t$ be the convex body enclosed by $M_t$. By Blaschke
selection theorem, there is a sequence in $\{K_{t_k}\}$, which is
still denoted by $\{K_{t_k}\}$, such that
\begin{equation*}
K_{t_k}\rightarrow \widetilde{K} \quad \text{as} \quad k\rightarrow +\infty.
\end{equation*}
Since $K_{t_k}$ is an origin-symmetric convex body, $\widetilde{K}$ is also origin-symmetric.
Then
\begin{eqnarray*}
\min_{\uS}h_{\widetilde{K}}=\lim_{k\rightarrow+\infty}\min_{\uS}h_{K_{t_k}}
=0.
\end{eqnarray*}
It follows that $\widetilde{K}$ is contained in a hyperplane in
$\mathbb{R}^{n}$. Then
$$\rho_{\widetilde{K}}=0, \quad \text{a.e. in $\uS$}.$$
Using Lemma \ref{lem03}, we have for any $\epsilon>0$
\begin{eqnarray*}
V_g(M_0)&=&
V_{g}(M_{t_k})\\
&\leq& \lim_{k\rightarrow+\infty}\int_{\uS} g(u)\log [\rho(u,t_k)+\epsilon]du\\
&=& \int_{\uS}g(u)\log \epsilon du\\
&=&C \log \epsilon \rightarrow-\infty \quad \text{as} \quad \epsilon\rightarrow 0 ,
\end{eqnarray*}
which is a contradiction. Then $$\min_{\uS\times[0,T)}h(x,t)\geq C$$
for some positive constant $C$. So we complete the proof.
\end{proof}

Due to the convexity of $M_t$, Lemma \ref{lem04} also implies the gradient
estimates of $h(\cdot,t)$ and $\rho(\cdot,t)$.

\begin{lemma}\label{lem06}
Let $X(\cdot, t)$ be a strictly convex solution to the flow
\eqref{flow} which encloses the origin for $t \in [0,T) $, then we
have
\begin{gather*}
  |\nabla h(x,t)| \leq C,
  \quad \forall (x,t) \in \mathbb{S}^{n-1} \times [0, T), \\
  |\nabla \rho(u,t)| \leq C,
  \quad \forall (u,t) \in \mathbb{S}^{n-1} \times [0, T),
\end{gather*}
where $C$ is a positive constant depending only on the constant in Lemma \ref{lem04}.
\end{lemma}

\begin{proof}
By virtue of \eqref{eq:8}, we have
\begin{equation}\label{G-1}
  \rho^2=|\nabla h|^2+h^2\geq |\nabla h|^2.
\end{equation}
By \eqref{h}, \eqref{eq:8} and \eqref{eq:9}, we have
\begin{equation}\label{G-2}
h= \frac{\rho^2}{\sqrt{\rho^{2}+|\nabla \rho|^{2}}}\leq
\frac{\rho^2}{|\nabla \rho|}.
\end{equation}
Using the two inequalities \eqref{G-1} and \eqref{G-2}, the
estimates of this lemma now follows directly from Lemma \ref{lem04}.
\end{proof}

\section{Uniform bounds for principal curvatures}

In this section, we continue to establish uniform upper and lower
bounds for principal curvatures. These estimates can be obtained by
considering proper auxiliary functions, see
\cite{CHZ.MA.373-2019.953,CLLN, LSW.JEMSJ.22-2020.893,
LL.TAMS.373-2020.5833} for similar techniques. First, we need the
following lemma.

\begin{lemma}\label{G-g}
Given two positive constants $r<R$. For the function
$$G(y)=\frac{|y|^n}{g(\frac{y}{|y|})},\quad y=(y^1, ..., y^n)\in A(r, R)=\{y \in \mathbb{R}^{n}: r<|y|<R\},$$
we have
\begin{eqnarray*}
\norm{G}_{C^k(A(r, R))}\leq C_k\norm{g}_{C^k(\uS)},
\end{eqnarray*}
where $k=0,1,2$, and $C_k$ is a positive constant depending only on $n, r, R$,
$\norm{g}_{C^k(\uS)}$, and $\min\limits_{\uS} g$.
\end{lemma}

\begin{proof}
Denote $\partial_i=\frac{\partial}{\partial y^i}$ for $1\leq i\leq
n$. It is clearly see that
\begin{eqnarray*}
\overline{\nabla}_i|y|^n=n|y|^{n-2}y^i \quad \text{and} \quad
\overline{\nabla}_ig=\frac{1}{|y|^3}\langle \overline{\nabla} g,|y|^2\partial_i-yy^i\rangle.
\end{eqnarray*}
Thus,
\begin{eqnarray*}
\overline{\nabla}_iG(y)
&=&\frac{n|y|^{n-2}y^i}{g}-\frac{|y|^{n-3}\langle \overline{\nabla} g,|y|^2\partial_i-yy^i\rangle}{g^2}.
\end{eqnarray*}
It follows consequently
\begin{eqnarray*}
\norm{G}_{C^1(A(r, R))}\leq C_1 \norm{g}_{C^1(\mathbb{S}^n)}.
\end{eqnarray*}

Moreover, we have
\begin{eqnarray*}
\overline{\nabla}_j\overline{\nabla}_i|y|^n=n(n-2)|y|^{n-4}y^iy^j
+n|y|^{n-2}\delta_{ij}
\end{eqnarray*}
and
\begin{eqnarray*}
\overline{\nabla}_j\overline{\nabla}_ig
&=&\frac{1}{|y|^3}\overline{\nabla}^2 g\Big(|y|^2\partial_i-yy^i, |y|^2\partial_i-yy^i\Big)
+\frac{1}{|y|^3}\Big\langle \overline{\nabla} g,2y^j\partial_i-y^i\partial_j-y\delta_{ij}\Big\rangle
\\&&-\frac{3y^j}{|y|^5}\langle \overline{\nabla} g,|y|^2\partial_i-yy^i\rangle.
\end{eqnarray*}
Note that
\begin{eqnarray*}
\overline{\nabla}_j\overline{\nabla}_iG
&=&\frac{1}{g}\overline{\nabla}_j\overline{\nabla}_i|y|^n
-|y|^n\frac{\overline{\nabla}_j\overline{\nabla}_ig}{g^2}+2|y|^n
\frac{\overline{\nabla}_ig\overline{\nabla}_j g}{g^3}-2\overline{\nabla}_i(|y|^n)
\frac{\overline{\nabla}_j g}{g}.
\end{eqnarray*}
Thus, we get
\begin{eqnarray*}
\norm{G}_{C^2(A(r, R))}\leq C_2 \norm{g}_{C^2(\mathbb{S}^n)}
\end{eqnarray*}
in view of
\begin{eqnarray*}
\overline{\nabla}^2 g(e_i, e_j)=\nabla^2 g(e_i,
e_j)-\langle\overline{\nabla}g, \frac{y}{|y|}\rangle \delta_{ij}
\end{eqnarray*}
for a local orthonormal frame $\{e_{1}, \cdots, e_{n-1}\}$ on
$\mathbb{S}^{n-1}$. So, our proof is completed.
\end{proof}

If we have proved Lemma \ref{G-g}, we can derive the uniform upper
bound of the Gauss curvature of $M_t$ and the uniform lower bound
for principal curvatures by similar calculations which have been
done in \cite{CLLN}. In the rest of this section, we take a local
orthonormal frame $\{e_{1}, \cdots, e_{n-1}\}$ on $\mathbb{S}^{n-1}
$ such that the standard metric on $\mathbb{S}^{n-1} $ is
$\{\delta_{ij}\}$. And double indices always mean to sum from $1$ to
$n-1$.

\begin{lemma}\label{lem07}
Let $X(\cdot, t)$ be a strictly convex solution to the flow
\eqref{flow} which encloses the origin for $t \in [0,T) $, then we
have
\begin{equation*}
  \mathcal{K}(x,t) \leq C,
  \quad \forall (x,t) \in \mathbb{S}^{n-1} \times [0, T),
\end{equation*}
where $C$ is a positive constant depending only on the constants in
Lemmas \ref{lem04} and \ref{lem06}.
\end{lemma}

\begin{proof}
Set
\begin{equation*}
Q(x, t)
=\frac{-\pd_th(x,t)+h(x,t)}{h(x,t)-\varepsilon_0}=\frac{f(x)\rho^n(u)}{(h-\varepsilon_0)g(u)}
\mathcal{K}(x,t),
\end{equation*}
where
\[ \varepsilon_0 =\frac{1}{2}\,\inf_{\uS\times[0,T)} h(x, t) \]
and the second equality follows from \eqref{ht}. For each
$t\in[0,T)$, assume $Q(\cdot,t)$ attains its maximum at some point
$x_t\in\uS$. At $(x_t, t)$, we can obtain
\begin{equation}\label{K-1}
  0=Q_i
  =\frac{-\pd_t h_i+h_i}{h-\varepsilon_0}
  +\frac{\pd_th- h}{(h-\varepsilon_0)^2}h_i,
\end{equation}
and
\begin{equation}\label{K-2}
  0\geq Q_{ij}
  =\frac{-\pd_th_{ij}+h_{ij}}{h-\varepsilon_0}
  +\frac{(\pd_th-h)h_{ij}}{(h-\varepsilon_0)^2},
\end{equation}
where \eqref{K-1} is used in \eqref{K-2}. Recall that
$b_{ij}=h_{ij}+h\delta_{ij}$, and $b^{ij}$ is its inverse matrix.
Using the inequality \eqref{K-2}, it yields
\begin{equation*}
  \begin{split}
    \pd_tb_{ij}
    &=\pd_th_{ij}+\pd_th\delta_{ij} \\
    &\geq h_{ij} +\frac{\pd_th-h}{h-\varepsilon_0}h_{ij} +\pd_th\delta_{ij} \\
    &= b_{ij} - Q (b_{ij} -\varepsilon_0\delta_{ij}).
\end{split}
\end{equation*}
Noticing that $\mathcal{K}=1/\det(b_{ij})$, we have
\begin{equation}\label{eq:12}
  \begin{split}
    \pd_t\mathcal{K}
    &= -\mathcal{K}b^{ji}\pd_tb_{ij} \\
    &\leq -\mathcal{K}b^{ji}[b_{ij} - Q (b_{ij} -\varepsilon_0\delta_{ij})] \\
    &= -\mathcal{K}\bigl[(n-1)(1 - Q) + Q\varepsilon_0 \TR(b^{ij})\bigr].
  \end{split}
\end{equation}
Note that
\begin{equation}\label{eq:10}
Q(x,t)
  = \frac{f(x)\rho^n(u)}{(h-\varepsilon_0)g(u)} \mathcal{K}(x,t),
\end{equation}
we derive by Lemma \ref{lem04}
\begin{equation} \label{eq:11}
  \frac{1}{C_1} Q(x,t)\leq \mathcal{K}(x,t)\leq C_1 Q(x,t),
\end{equation}
where $C_1$ is a positive constant depending only on the constant
$C$ in Lemma \ref{lem04}, and the upper and lower bounds of $f,g$ on
$\uS$. Combining Lemma \ref{lem04} and the inequalities
\eqref{eq:12} and \eqref{eq:11}, we have
\begin{equation}\label{eq:15}
  \begin{split}
    \pd_t\mathcal{K}
    &\leq (n-1)  \mathcal{K} Q -(n-1)\varepsilon_0 Q \mathcal{K}^{\frac{n}{n-1}} \\
    &\leq C_2 Q^2 -C_3 Q^{\frac{2n-1}{n-1}},
  \end{split}
\end{equation}
where the inequality
\begin{equation*}
  \frac{1}{n-1}\TR(b^{ij})
  \geq \det(b^{ij})^{\frac{1}{n-1}}
  =\mathcal{K}^{\frac{1}{n-1}}
\end{equation*}
is used. Here $C_2$, $C_3$ are positive constants depending only on
$n$, $\varepsilon_0$ and $C_1$. By the definition of $Q$ and
\eqref{K-1},
\begin{equation}\label{eq:5}
\pd _t h_i=(1-Q)h_i.
\end{equation}
Thus, we have
\begin{eqnarray}\label{eq:6}
\pd_t (\nabla h+hx) &=& \pd_t (h_ie_i+hx)\\ \nonumber &=& \pd_t
h_ie_i+ (\pd_t h) x\\\nonumber &=&
(1-Q)h_ie_i+(h-(h-\varepsilon_0)Q)x\\\nonumber &=& (1-Q)(\nabla
h+hx)+\varepsilon_0 Q x.
\end{eqnarray}
So, we can say that
\begin{eqnarray}\label{eq:7}
\pd_t G(\nabla h+hx)&=&\langle \overline{\nabla}G, \pd_t (\nabla
h+hx)\rangle\\ \nonumber
&=&(1-Q)\langle\overline{\nabla}G, \nabla h+hx\rangle+\varepsilon_0 Q\langle \overline{\nabla}G, x\rangle\\
\nonumber &\leq& (1-Q)|\overline{\nabla}G| \cdot |\nabla
h+hx|+\varepsilon_0 Q|\overline{\nabla}G|\\ \nonumber &\leq&
C_4(1-Q) |\nabla h+hx|+C_4\varepsilon_0 Q,
\end{eqnarray}
where we know by Lemma \ref{G-g} that $C_4$ is a positive constant
depending on $n$, the constant $C$ in Lemmas \ref{lem04} and
\ref{lem06}, $\norm{g}_{C^1(\uS)}$ and $\min\limits_{\uS} g$.

Thus,
\begin{equation}\label{eq:16}
  \begin{split}
    &\frac{\pd}{\pd t} \Bigl[\frac{G(\nabla h+hx)}{h(x,t)-\varepsilon_0}  \Bigr] \\
    =& \frac{\pd_t G}{h-\varepsilon_0}-\frac{G\pd_t h}{(h-\varepsilon_0)^2}\\
    \leq &\frac{(1-Q)|\overline{\nabla}G| \cdot |\overline{\nabla} h|
    +\varepsilon_0 Q|\overline{\nabla}G|}{h-\varepsilon_0}+\frac{[(1-Q)h+\varepsilon_0Q]G}{(h-\varepsilon_0)^2}\\
    \leq& C_5+C_5Q,
  \end{split}
\end{equation}
where $C_5$ is a positive constant depending only on
$\varepsilon_0$, $C_4$ and the constant $C$ in Lemmas \ref{lem04} and \ref{lem06}.

By virtue of \eqref{eq:10}, \eqref{eq:15} and \eqref{eq:16}, we have
at $(x_t,t)$
\begin{eqnarray}\label{eq:14}
\pd_t Q&=& f\pd_t[\frac{G}{h-\varepsilon_0}]\mathcal{K}+f
\frac{G}{h-\varepsilon_0}\pd_t \mathcal{K}\\ \nonumber &\leq&
\max_{\uS} f \cdot C_5(1+Q)C_1 Q+\frac{Q}{\mathcal{K}}(C_2 Q^2 -C_3
Q^{\frac{2n-1}{n-1}})\\ \nonumber &\leq& \max_{\uS} f
\cdot C_1C_5Q(1+Q)+ C_1C_2 Q^2 -C_1^{-1}C_3 Q^{\frac{2n-1}{n-1}})\\
\nonumber &\leq& CQ+CQ^2-CQ^{\frac{2n-1}{n-1}},
\end{eqnarray}
where we have used \eqref{eq:11} to obtain the third inequality and
$C$ is a positive constant depending only on $\max\limits_{\uS} f$,
$C_1$, $C_2$, $C_3$ and $C_5$. Thus, whenever $Q(x_t,t)$ is greater
than some constant which is independent of $t$, we have
\begin{equation*}
\pd_tQ<0,
\end{equation*}
which implies that $Q$ has an uniform upper bound. By \eqref{eq:11},
$\mathcal{K}$ has a uniform upper bound.
\end{proof}

\begin{lemma}\label{lem08}
Let $X(\cdot, t)$ be a strictly convex solution to the flow
\eqref{flow} which encloses the origin for $t \in [0,T) $, then for
the principal curvatures $\kappa_{i}(x,t)$ of $M_t$, we have
  \begin{equation*}
  \kappa_{i}(x,t) \geq C, \quad \forall (x,t)\in\uS\times[0,T), \quad \forall i=1,\cdots,
  n-1,
  \end{equation*}
where $C$ is a positive constant depending only on the constant in
Lemmas \ref{lem04} and \ref{lem06}.
\end{lemma}

\begin{proof}
Set
\begin{equation*}
  \widetilde{\Lambda}(x, t)
  =\log \lambda_{\max}(b_{ij})-A\log h+B |\nabla h|^2,
  \quad \forall (x,t)\in\uS\times[0,T),
\end{equation*}
where $b_{ij}=h_{ij}+h\delta_{ij}$ as before,
$\lambda_{\max}(b_{ij})$ denotes the maximal eigenvalue of the
matrix $(b_{ij})$, and $A$ and $B$ are positive constants to be
chosen later.

For any fixed $T'\in(0,T)$, assume $\max_{\uS\times[0,T']}
\widetilde{\Lambda}(x,t)$ is attained at some point
$(x_0,t_0)\in\uS\times[0,T']$. By choosing a suitable orthonormal
frame, we may assume
\begin{equation*}
\{b_{ij}(x_0, t_0)\} \quad \text{is diagonal and} \quad
\lambda_{\max}(b_{ij})(x_0, t_0)=b_{11}(x_0, t_0).
\end{equation*}
Thus, the new function defined on $\uS\times[0,T']$
\begin{equation*}
\Lambda(x, t)=\log b_{11}-A\log h+B |\nabla h|^2
\end{equation*}
also attains its maximum at $(x_0,t_0)$. Thus, we have at
$(x_0,t_0)$
\begin{equation}\label{C2-1d}
0=\Lambda_i=b^{11}b_{11; i}-A\frac{h_i}{h}+2B \sum_{k}h_k h_{ki},
\end{equation}
and
\begin{equation} \label{C2-2d}
  \begin{split}
    0\geq \Lambda_{ij}
    &=b^{11}b_{11; ij}-(b^{11})^2 b_{11; i}b_{11; j}
    \\
    &\hskip1.1em-A\Bigl(\frac{h_{ij}}{h}-\frac{h_i h_j}{h^2}\Bigr)  +2B \sum_{k}(h_{kj} h_{ki}+h_kh_{kij}),
  \end{split}
\end{equation}
where $(b^{ij})$ is the inverse of the matrix $(b_{ij})$. Without
loss of generality, we can assume $t_0>0$. Then, we get at
$(x_0,t_0)$
\begin{equation} \label{eq:13}
  \begin{split}
  0\leq\pd_t\Lambda
  &=b^{11}(\pd_th_{11}+\pd_th) -A\frac{\pd_th}{h} +2B\sum_k h_k\pd_th_k.
  \end{split}
\end{equation}
From the equation \eqref{ht}, we know
\begin{equation}\label{c1}
 \log(h- \pd_th)
  = \log \mathcal{K}(x,t) + \log f(x)G(\nabla h+hx).
\end{equation}
Set
\begin{equation*}
  w(x,t)
  = \log\Big[ f(x)G(\nabla h+hx)\Big],
\end{equation*}
where $$G(\nabla h+hx)=\frac{|\nabla h+h x|^{n}}{g\bigg(\frac{\nabla h+hx}{|\nabla h+hx|}\bigg)}.$$ Differentiating \eqref{c1} gives
\begin{equation}\label{C2-00}
\frac{h_k-\pd_th_k}{h-\pd_th}=-b^{ji}b_{ij; k}+w_k,
\end{equation}
and
\begin{equation}\label{C2-11}
  \frac{h_{11}-\pd_th_{11}}{h-\pd_th}
  =\frac{(h_{1}-\pd_th_{1})^2}{(h-\pd_th)^2}
  -b^{ii}b_{ii; 11} +b^{ii}b^{jj}(b_{ij; 1})^2 +w_{11}.
\end{equation}
Multiplying both sides of \eqref{C2-11} by $-b^{11}$, it yields
\begin{equation}\label{eq:19}
  \begin{split}
  \frac{b^{11}\pd_th_{11}-b^{11}h_{11}}{h-\pd_th}
  &\leq b^{11}b^{ii}b_{ii; 11} -b^{11}b^{ii}b^{jj}(b_{ij; 1})^2 -b^{11}w_{11} \\
  &\leq b^{11}b^{ii}b_{11; ii} -b^{11}b^{ii}b^{11}(b_{i1; 1})^2 -\sum_i b^{ii} \\
  &\hskip1.1em +b^{11}(n-1-w_{11}),
  \end{split}
\end{equation}
where we use the Ricci identity $ b_{ii; 11}=b_{11;
ii}-b_{11}+b_{ii}$. We know that $b^{ij}\Lambda_{ij}\leq0$ from
\eqref{C2-2d} which implies at $(x_0, t_0)$
\begin{equation*}
  \begin{split}
    b^{11}b^{ii}b_{11; ii}-(b^{11})^2b^{ii} (b_{11; i})^2
    \leq
    Ab^{ii}\Bigl(\frac{h_{ii}}{h}-\frac{h_i^2}{h^2}\Bigr)
    -2B \sum_{k}b^{ii}(h_{ki}^2+h_kh_{kii}).
  \end{split}
\end{equation*}
Thus,
\begin{equation}\label{eq:199}
  \begin{split}
    &b^{11}b^{ii}b_{11; ii}-(b^{11})^2b^{ii} (b_{11; i})^2
    \\ \leq&
    \frac{(n-1)A}{h} -A \sum_i b^{ii}
    -\frac{Ab^{ii}h_i^2}{h^2} +4(n-1)Bh -2B\sum_i b_{ii} -2Bh^2\sum_i b^{ii} \\
    &-2Bb^{ii}h_kb_{ii;k} +2Bb^{ii}h_i^2,
  \end{split}
\end{equation}
where we use the following equalities
\begin{gather*}
  b^{ii}h_{ii}
  =b^{ii}(b_{ii}-h)
  =n-1-h\sum_i b^{ii}, \\
  \sum_{k}b^{ii}h_{ki}^2
  =b^{ii}h_{ii}^2
  =b^{ii}(b_{ii}^2-2hb_{ii}+h^2)
  = -2(n-1)h+\sum_i b_{ii} +h^2\sum_i b^{ii},  \\
  \sum_{k}b^{ii}h_kh_{kii}
  =\sum_{k}b^{ii}h_k(b_{ki;i}-h_i\delta_{ki})
  =\sum_{k}b^{ii}h_kb_{ii;k} -b^{ii}h_i^2.
\end{gather*}
Here the fact that $b_{ij;k}$ is symmetric in all indices is used to
get the third equality above. Inserting the inequality
\eqref{eq:199} into \eqref{eq:19}, we obtain that
\begin{equation}\label{eq:20}
  \begin{split}
    \frac{b^{11}\pd_th_{11}-b^{11}h_{11}}{h-\pd_th}
    &\leq
    \frac{(n-1)A}{h} -(A+2Bh^2+1) \sum_i b^{ii}
    -\frac{A-2Bh^2}{h^2}b^{ii}h_i^2 \\
    &\hskip1.2em  +4(n-1)Bh -2B\sum_i b_{ii} \\
    &\hskip1.2em -2B\sum_{k}b^{ii}h_kb_{ii;k} +b^{11}(n-1-w_{11}).
  \end{split}
\end{equation}
Using \eqref{C2-00}, we get
\begin{equation}\label{eq:18}
\frac{2B\sum_k h_k\pd_th_k}{h-\pd_th} = \frac{2B|\nabla
h|^2}{h-\pd_th} +2B\sum_{k}b^{ii}h_kb_{ii;k} -2B\langle \nabla
h,\nabla w \rangle.
\end{equation}
Now dividing \eqref{eq:13} by $h-\pd_th$ gives
\begin{equation*}
  \begin{split}
    0 &\leq \frac{b^{11}(\pd_th_{11}-h_{11}+b_{11}-h+\pd_th)}{h-\pd_th}
    -\frac{A\pd_th}{h(h-\pd_th)}
    +\frac{2B\sum_k h_k\pd_th_k}{h-\pd_th}  \\
    &= \frac{b^{11}(\pd_th_{11}-h_{11})}{h-\pd_th}
    +\frac{2B\sum_k h_k\pd_th_k}{h-\pd_th}
    -b^{11} +\frac{A}{h}-\frac{A-1}{h-\pd_th},
  \end{split}
\end{equation*}
which together with \eqref{eq:20} and \eqref{eq:18} implies that
\begin{equation*}
  \begin{split}
    0
    &\leq \frac{nA}{h} -(A+2Bh^2+1) \sum_i b^{ii}
    -\frac{A-2Bh^2}{h^2}\sum_i b^{ii}h_i^2 \\
    &\hskip1.2em  +4(n-1)Bh -2B\sum_i b_{ii}  +b^{11}(n-2-w_{11}) \\
    &\hskip1.2em -2B\langle \nabla h,\nabla w \rangle
    -\frac{A-1-2B|\nabla h|^2}{h-\pd_th}.
  \end{split}
\end{equation*}
Now we choose $A=n+2BC^2$, where $C$ is the constant in Lemma
\ref{lem04}, we can obtain
\begin{equation}\label{eq:22}
  (A-n+3) \sum_i b^{ii}
  +2B\sum_i b_{ii}
  \leq
  C_1(A+B)
  -b^{11}w_{11}
  -2B\langle \nabla h,\nabla w \rangle,
\end{equation}
where $C_1$ is a positive constant depending only on $n$ and the constant $C$
in Lemma \ref{lem04}.

A direct calculation gives
\begin{eqnarray*}
  \nabla_i G(\nabla h+hx) = \langle\overline{\nabla} G,\overline{\nabla}_i \overline{\nabla}h\rangle
  =\langle\overline{\nabla}G,e_k \rangle b_{ik},
\end{eqnarray*}
and
\begin{eqnarray*}
    \nabla_{i}\nabla_{j} G (\nabla h+hx)&=&\overline{\nabla}^2 G\Big( \overline{\nabla}_j(\overline{\nabla}h)
    \rangle, \overline{\nabla}_i(\overline{\nabla}h)\Big)+
    \Big\langle \overline{\nabla} G, \overline{\nabla}_{i}\overline{\nabla}_{j} (\overline{\nabla}h)\Big\rangle\\
    &=&\overline{\nabla}^2 G\Big(\sum_{k}b_{ik}e_k, \sum_{l}b_{jl}e_l\Big)-\langle
    \overline{\nabla} G ,x \rangle b_{ij}+\sum_{k}\langle \overline{\nabla} G ,e_k \rangle b_{ij; k}.
\end{eqnarray*}
Thus, we have
\begin{equation*}
  \begin{split}
    -2B\langle \nabla h,\nabla w \rangle
    &= -2B \sum_k h_k \left(
      \frac{f_k}{f}
      +\frac{\sum_{l}\langle\overline{\nabla}G,e_l \rangle b_{kl}}{G}
    \right) \\
    &\leq C_2B
    +\sum_k \frac{2Bh_k\langle\overline{\nabla}G,e_k \rangle b_{kk}}{G},
  \end{split}
\end{equation*}
where $C_2$ is a positive constant depending only on the constants
$C$ in Lemma \ref{lem04}, $\norm{f}_{C^1(\uS)}$ and
$\min\limits_{\uS} f$. Moreover, we have by Lemma \ref{G-g}
\begin{eqnarray*}
    -w_{11}
    &=& \frac{f_1^2}{f^2} -\frac{f_{11}}{f}
    +\frac{[\langle\overline{\nabla}G,e_1 \rangle b_{11}]^2}{G^2}\\
    &&-\frac{1}{G}\Big[\overline{\nabla}^2 G(e_1,e_1) b^2_{11}-\langle \overline{\nabla} G ,x \rangle b_{11}
    +\langle \overline{\nabla} G ,e_k \rangle b_{11,k}\Big]\\
    &\leq& C_3(1+b_{11}+b_{11}^2) +\frac{\langle \overline{\nabla} G ,e_k \rangle b_{11,k}}{G},
\end{eqnarray*}
where $C_3$ is a positive constant depending only on the constants
$C$ in Lemmas \ref{lem04} and \ref{lem06}, $\norm{f}_{C^2(\uS)}$,
$\norm{g}_{C^2(\uS)}$, $\min\limits_{\uS} f$ and $\min\limits_{\uS}
g$. Therefore, combining the two inequalities above, we have
\begin{equation}\label{eq:21}
  \begin{split}
    &-b^{11}w_{11} -2B\langle \nabla h,\nabla w \rangle
   \\ \leq& C_3(b^{11}+1+b_{11}) + C_2B+\sum_{k}\frac{\langle \overline{\nabla}G,e_k\rangle}{G} (b^{11}b_{11;k} + 2Bh_k h_{kk}) \\
    =& C_3(b^{11}+1+b_{11}) + C_2B
    +\sum_{k}\frac{\langle \overline{\nabla}G,e_k\rangle}{G}\cdot \frac{Ah_k}{h} \\
    \leq& C_3(b^{11}+1+b_{11}) + C_2B +C_4A,
  \end{split}
\end{equation}
where we have used the equality \eqref{C2-1d}, and $C_4$ is a
positive constant depending only on the constants $C$ in Lemma
\ref{lem04} and $n$. Inserting \eqref{eq:21} into \eqref{eq:22}, we
have
\begin{equation*}
  (A-n) \sum_i b^{ii}
  +2B\sum_i b_{ii}
  \leq
  (C_1+C_4)A
  +(C_1+C_2)B
  +C_3(b^{11}+1+b_{11}),
\end{equation*}
which together with $A=n+2BC^2$ implies that
\begin{equation*}
  \begin{split}
  &(2BC^2-C_3) \sum_i b^{ii}
  +(2B-C_3)\sum_i b_{ii} \\
  \leq&
  (C_1+C_4)(n+2BC^2)
  +(C_1+C_2)B
  +C_3.
  \end{split}
\end{equation*}
If we choose $B=\max\set{\frac{C_3+1}{2}, \frac{C_3+1}{2C^2}}$, we
see that $b_{11}(x_0,t_0)$ is bounded from above by a positive
constant depending only on $n$, $C_1$, $C_2$, $C_3$ and $C_4$. Using
Lemmas \ref{lem04}, $\widetilde{\Lambda}(x_0,t_0)$ is bounded from
above by a positive constant depending only on $n$, $C$, $C_1$,
$C_2$, $C_3$ and $C_4$. Thus, we prove the conclusion of this lemma,
by noticing that $T'$ can be any number in $(0,T)$.
\end{proof}

\section{Existence of solutions}

In this section, we will complete the proof of Theorem \ref{thm2}.
Combining Lemma \ref{lem07} and Lemma \ref{lem08}, we see that the
principal curvatures of $M_{t}$ has uniform positive upper and lower
bounds. This together with Lemmas \ref{lem04} and \ref{lem06}
implies that the evolution equation \eqref{ht} is uniformly
parabolic on any finite time interval. Thus, using Krylov- Safonov
estimates \cite{KS.IANSSM.44-1980.161} and Schauder estimates of the
parabolic equations, we can say that the smooth solution of
\eqref{ht} exists for all time. And by these estimates again, a
subsequence of $M_t$ converges in $C^\infty$ to a positive, smooth,
uniformly convex hypersurface $M_\infty$ in $\R^n$. Now to complete
the proof of Theorem \ref{thm2}, it remains to check the support
function of $M_\infty$ satisfies Eq. \eqref{dOMP-f}.

Let $\tilde{h}$ be the support function of $M_\infty$. We need to prove that
$\tilde{h}$ is a solution to the following equation
\begin{equation} \label{dOMP-f1}
g\bigg(\frac{\nabla h+hx}{|\nabla h+hx|}\bigg) |\nabla h+hx|^{-n}h
\det(\nabla^2h +hI) =f \text{ on } \uS.
\end{equation}
By Lemma \ref{lem02}, $J'(t)\leq0$ for any $t>0$. Since
\begin{equation*}
\int_0^t [-J'(t)] \dd t =J(0)-J(t) \leq C \int_{\uS} f dx,
\end{equation*}
we get
\begin{equation*}
\int_0^\infty [-J'(t)] \dd t  \leq C \int_{\uS} f dx.
\end{equation*}
Thus, there exists a subsequence of times $t_j\to\infty$ such that
\begin{equation*}
-J'(t_j) \to 0 \text{ as } t_j\to\infty.
\end{equation*}
Thus, we obtain using \eqref{eq:23}
\begin{equation*}
\int_{\uS} \frac{1}{g\tilde{h}\widetilde{\mathcal{K}} \tilde{\rho}^n}(g\tilde{h}-f\widetilde{\mathcal{K}}\tilde{\rho}^n)^2 dx=0,
\end{equation*}
where $\widetilde{\mathcal{K}}$ is the Gauss curvature of $M_\infty$. It implies that
\begin{equation*}
  g(u) \frac{\tilde{h}(x)}{\widetilde{\mathcal{K}}\tilde{\rho}^n(u)}=f(x) \quad \mbox{on}\quad \uS,
\end{equation*}
which means $\tilde{h}$ is a solution to the equation
\eqref{dOMP-f1}.


\end{document}